\documentclass[12pt,a4paper,twoside]{article}
\usepackage{hyperref}
\usepackage{amsmath,amsfonts,amsthm,amsopn,color,amssymb}
\usepackage{palatino}
\usepackage{graphicx}

\newcommand{\eps}{\varepsilon}

\newcommand{\R}{\mathbb{R}}

\newcommand{\RN}{{\mathbb{R}^N}}
\newcommand{\RT}{{\mathbb{R}^3}}

\renewcommand{\le}{\leqslant}
\renewcommand{\ge}{\geqslant}
\renewcommand{\a }{\alpha }

\renewcommand{\d }{\delta }

\renewcommand{\l }{\lambda}
\newcommand{\n }{\nabla }
\newcommand{\s }{\sigma }
\renewcommand{\t}{\theta}
\renewcommand{\O}{\Omega}

\renewcommand{\S}{\Sigma}

\newcommand{\OO}{\mathcal O}
   \newcommand{\TT}{\mathcal T}

\newcommand{\Jtq}{J^T_q}

\renewcommand{\H}{H^1(\RT)}
\newcommand{\Hr}{H^1_r(\RT)}
\newcommand{\Hco}{H^1_{cyl,o}(\RT)}

\newcommand{\E}{\mathcal{E}}

\newcommand{\M}{\mathcal{M}}

\renewcommand{\o}{\omega}
\newcommand{\mtq}{m_q^T}
\newcommand{\mutq}{\mu^{T,q}}

\newcommand{\D }{{\mathcal D}^{1,2}(\RT)}
\newcommand{\Dce }{{\mathcal D}^{1,2}_{cyl,e}(\RT)}

\newcommand{\irt }{\int_{\RT}}

\newtheorem{theorem}{Theorem}[section]
\newtheorem{lemma}[theorem]{Lemma}

\renewenvironment{proof}{\noindent{\textbf{Proof\quad}}}{$\hfill\square$\vspace{0.2 cm}\\}
\newenvironment{proofmain}{\noindent{\textbf{Proof of Theorem  \ref{main}\quad}}}{$\hfill\square$\vspace{0.2 cm}\\}
\newenvironment{proofsol}{\noindent{\textbf{Proof of Theorem  \ref{sol}\quad}}}{$\hfill\square$\vspace{0.2 cm}\\}

\title{{\bf Concentration and compactness\\
in nonlinear Schr\"odinger-Poisson system\\
with a general nonlinearity\footnote{The author is supported by
M.I.U.R. - P.R.I.N. ``Metodi variazionali e topologici nello
studio di fenomeni non lineari''}}}

\author{A. Azzollini \thanks{Dipartimento di Matematica ed Informatica, Universit\`a degli
Studi della Basilicata,  Via dell'Ateneo Lucano 10, I-85100
Potenza, Italy, e-mail: {\tt antonio.azzollini@unibas.it}} }

\date{}

\begin{document}
    \maketitle

\begin{abstract}
In this paper we use a concentration and compactness argument to
prove the existence of a nontrivial nonradial solution to the
nonlinear Schr\"odinger-Poisson equations in $\RT,$ assuming on
the nonlinearity the general hypotheses introduced by Berestycki
\& Lions.\\

%{\bf Keywords} {Schr\"odinger-Poisson system, concentration and
%compactness, general nonlinearity}
\end{abstract}

\section*{Introduction}
We consider the following Schr\"odinger-Poisson system
\begin{equation}
\left\{
\begin{array}{ll}
-\Delta u+q\phi u=g(x,u)&\hbox{in }\O,
\\
-\Delta \phi=q u^2&\hbox{in }\O,
\end{array}
\right.
\end{equation}
where $\O$ is an unbounded domain in $\RT$ and
$g:\RT\times\R\to\R$. In \cite{ADP} the system has been studied
using a variational approach, for $\O=\RT$ and assuming on
$g=g(u)$ the Berestycki and Lions hypotheses (see \cite{BL1}). In
particular, it has been showed that the solutions can be found as
critical points of an associated functional defined in $\H$. A
first difficulty in applying the classical methods of critical
points theory is the lack of compactness, due to the unboundedness
of the domain. In \cite{ADP} this difficulty has been overcome by
restricting the functional to the natural constraint $\Hr$, the
set of the radially symmetric functions in $\H$, for which compact
embeddings hold.

However, it could happen that such a restriction is not allowed or
not suitable to our aim. For example, consider these three
situations:
    \begin{itemize}
        \item $\O$ is not radially symmetric with respect to a point,
        \item $g(\cdot, s)$ is not invariant under the action of
        the group of rotations (for example in presence of a
        breaking-symmetry potential),
        \item we are looking for non-radial solutions of the
        problem.
    \end{itemize}
Each of these situations does not allow us to use the set of the
radially symmetric functions as a nice functional setting, and we
have to handle the problem of the lack of compactness using a
different approach.

The aim of this paper is to show how the concentration and
compactness principle can be used as an alternative technique to
get compactness. In particular, in the same spirit of \cite{DA},
we are interested in looking for non-radial solutions to the
problem
    \begin{equation}    \label{SP}\tag{${\cal SP}$}
\left\{
\begin{array}{ll}
-\Delta u+q\phi u=g(u)&\hbox{in }\RT,
\\
-\Delta \phi=q u^2&\hbox{in }\RT.
\end{array}
\right.
\end{equation}

In \cite{DA} an existence result has been proved assuming that
$g(u)=|u|^{p-2}u$ and $4<p<6$. Here we consider a more general
nonlinear term, namely a Berestycki \& Lions type nonlinearity. So
we assume that
\begin{itemize}
\item[({\bf g1})] $g\in C(\R,\R)$, $g$ odd; \item[({\bf g2})]
$-\infty <\liminf_{s\to 0^+} g(s)/s\le \limsup_{s\to 0^+}
g(s)/s=-\o<0$; \item[({\bf g3})] $-\infty \le\limsup_{s\to
+\infty} g(s)/s^p\le 0$, $1<p<5$; \item[({\bf g4})] there exists
$\zeta>0$ such that $G(\zeta):=\int_0^\zeta g(s)\,d s>0$.
\end{itemize}

The literature on the Schr\"odinger-Poisson system in presence of
a pure power nonlinearity is very reach: we mention \cite{A,ADP}
and the references therein. In \cite{C1,CG,WZ}, also the linear
and the asymptotic linear case have been studied, whereas in
\cite{PiS1,PiS2,S} the problem has been studied in a bounded
domain. We refer to \cite{BF1} for more details on the physical
origin of this system.

Recently, the Schr\"odinger equation and the Schr\"odinger-Poisson
system in presence of a general nonlinear term have been
intensively studied by many authors. Using similar assumptions on
the nonlinearity $g$, \cite{AP,JT} and \cite{PS} studied,
respectively, a nonlinear Schr\"odinger equation in presence of an
external potential and a system of weakly coupled nonlinear
Schr\"odinger equations. The Schr\"odinger-Poisson system has been
considered in \cite{ADP}. We mention also \cite{BF2,M} where the
Klein-Gordon, Klein-Gordon-Maxwell and Schr\"odinger Poisson
equations have been considered in presence of the so called
``positive potentials''.

It is well known that the system \eqref{SP} is equivalent to an
equation containing a nonlocal nonlinear term. A non trivial
difficulty in applying concentration and compactness to this
equation in presence of a Berestycki \& Lions type nonlinearity,
consists in the fact that, since $g$ does not have any homogeneity
property, we can not use the usual arguments as in the pure power
case to avoid dichotomy (see \cite{AP08}). In order to overcome
this difficulty, we need to study the behaviour of the functional
associated to the problem with respect to rescaled functions.
However, when we rescale the variables, the behaviour of the
integral term coming from the nonlocal nonlinearity is such to
prevent us from using a direct approach. So we introduce a
modified functional, where a cut off function is introduced to
control the integral containing the {\it coupling term}. Finally,
we observe that, for $q$ small enough, the modified functional
corresponds with the original one computed on suitable minimizing
sequences. Observe that, for our analysis, it is fundamental the
invariance of the domain with respect to rescalements.

The main result of this paper is the following:
    \begin{theorem}\label{main}
        Assume ({\bf g1}),...,({\bf g4}). Then there exists $q>0$
        such that the system \eqref{SP} possesses a solution $(u,\phi)\in\H\times\D$
        with the
        following features
            \begin{enumerate}
                \item $u$ and $\phi$ are respectively odd and even with respect to the third
                variable,
                \item $u$ and $\phi$ are cylindrically symmetric with
                respect to the first two variables,
                \item $u$ is positive on the half space $x_3>0$
                (and, consequently, negative in the half space
                $x_3<0$), $\phi$ is positive everywhere.
            \end{enumerate}
    \end{theorem}

    The paper is organized as follows:

    in section \ref{sec:fun} we introduce the functional framework
    of the problem. In particular, we define a space of functions
    described by symmetry properties that no radial nontrivial function
    possesses. Then we reduce the study to a minimization problem.

    In section \ref{sec:comp}, we study the behaviour of the
    positive measures associated to the functions of a minimizing
    sequence, and we look for concentration on a bounded region.

    In section \ref{sec:proof} we provide the proof of the main
    theorem.

\section{The functional setting}\label{sec:fun}

We denote by $\H,$ $\D,$ $L^p(\RT)$ the usual Sobolev and Lebesgue
spaces with the respective norms:
    \begin{align*}
        \|u\| &=\left(\irt |\n u|^2 + u^2\right)^{\frac 12}\\
        \|u\|_{\D} &=\left(\irt |\n u|^2\right)^{\frac 12}\\
        \|u\|_p    &=\left(\irt |u|^p\right)^{\frac 1 p}.
    \end{align*} We first recall the following
well-known facts (see, for instance \cite{DM1}).
\begin{lemma}\label{le:prop}
For every $u\in \H$, there exists a unique $\phi_u\in \D$ solution
of
\[
-\Delta \phi=q u^2,\qquad \hbox{in }\RT.
\]
Moreover
\begin{itemize}
\item[i)] $\|\phi_u\|^2_{\D}=q\irt\phi_u u^2$; \item[ii)]
$\phi_u\ge 0$; \item[iii)] for any $\t>0$:
$\phi_{u_\t}(x)=\t^2\phi_u(x/\t)$, where $u_\t(x)=u(x/\t )$;
\item[iv)] there exist $C,C'>0$ independent of $u\in\H$ such that
$$\|\phi_u\|_{\D}\le C q \|u\|^2,$$
and
\begin{equation}\label{eq:phiq}
\irt\phi_u u^2\le C'q \|u\|^4.
\end{equation}
%\footnote{$\|\phi_u\|^2=q\int \phi_u u^2\leq q\|\phi_u\|_6\|u\|_\a^2\leq Cq\|\phi_u\|\|u\|_\a^2\Rightarrow\|\phi_u\|\leq Cq \|u\|_\a^2$\\
%and\\
%$\int \phi_u u^2\leq \|\phi_u\|_6\|u\|_\a^2\leq C\|\phi_u\|\|u\|_\a^2\leq C^2 q \|u\|_\a^4$}
\end{itemize}
\end{lemma}

Following \cite{BL1}, define $s_0:=\min\{s\in
[\zeta,+\infty[\;\mid g(s)=0\}$ ($s_0=+\infty$ if $g(s)\neq 0$ for
any $s\ge\zeta$) and set $\tilde g:\R\to\R$ the function such that
    \begin{equation}\label{eq:tilde}
      \tilde g(s)=\left\{
        \begin{array}{ll}
                g(s) &\hbox{ on } [0,s_0];
                \\
                0 &\hbox{ on } \R_+\setminus [0,s_0];
                \\
                -\tilde g(-s) &\hbox{ on } \R_-.
      \end{array}
      \right.
    \end{equation}
By the strong maximum principle and by ii) of Lemma \ref{le:prop},
a solution of \eqref{SP} with $\tilde g$ in the place of $g$ is a
solution of \eqref{SP}. So we can suppose that $g$ is defined as
in \eqref{eq:tilde}, so that ({\bf g1}), ({\bf g2}) and ({\bf g4})
hold, and we have also the following limit
    \begin{equation}\label{eq:limg}
        \lim_{s\to\infty} \frac{|g(s)|}{|s|^{p}}=0.
    \end{equation}
Moreover, we set for any $s\ge 0,$
    \begin{align*}
        g_1(s) & :=(g(s)+\o s)^+,
        \\
        g_2(s) & :=g_1(s)-g(s),
    \end{align*}
and we extend them as odd functions.
\\
Since
    \begin{align}
        \lim_{s\to 0}\frac{g_1(s)}{s} &= 0,\nonumber%\label{eq:lim1}
        \\
        \lim_{s\to\infty}\frac{g_1(s)}{|s|^{p}}&=0,\label{eq:lim2}
    \end{align}
and
    \begin{equation}
        g_2(s) \ge \o s,\quad  \forall s\ge 0,\label{eq:g2}
    \end{equation}
by some computations, we have that for any $\eps>0$ there exist
$C_\eps,$ $C'_\eps>0$ such that
    \begin{align}
        g_1(s) &\le C_\eps s^{p}+\eps s,\quad  \forall
        s\ge0\label{eq:g1g2a}\\
        g_1(s) &\le C'_\eps s^{5}+\eps s,\quad  \forall
        s\ge0\label{eq:g1g2abis}\\
        g_1(s) &\le C_\eps s^{p}+\eps g_2(s),\quad  \forall
        s\ge0\label{eq:g1g2}\\
        g_1(s) &\le C'_\eps s^{5}+\eps g_2(s),\quad  \forall
        s\ge0\label{eq:g1g2bis}.
    \end{align}
If we set
    \begin{equation*}
        G_i(t):=\int^t_0g_i(s)\,ds,\quad i=1,2,
    \end{equation*}
then, by \eqref{eq:g2}, we have
    \begin{equation}
         G_2(s) \ge \frac \o 2 s^2,\quad  \forall s\in\R\label{eq:G2}
    \end{equation}
and by \eqref{eq:g1g2a}, \eqref{eq:g1g2abis}, \eqref{eq:g1g2} and
\eqref{eq:g1g2bis}, for any $\eps>0$ there exists $C_\eps>0$ and
$C'_\eps>0$ such that
    \begin{align}
        G_1(s) &\le \frac {C_\eps} 6 |s|^{6}+\eps s^2,\quad  \forall
        s\in\R\nonumber\\
        G_1(s) &\le \frac {C'_\eps} {p+1} |s|^{p+1}+\eps s^2,\quad  \forall
        s\in\R\label{eq:G1G2a}\\
        G_1(s) &\le \frac {C_\eps} 6 |s|^{6}+\eps G_2(s),\quad  \forall
        s\in\R\label{eq:G1G2}\\
        G_1(s) &\le \frac {C'_\eps} {p+1} |s|^{p+1}+\eps G_2(s),\quad  \forall
        s\in\R\label{eq:G1G2b}.
    \end{align}

The solutions $(u,\phi)\in \H \times \D$ of \eqref{SP} are the
critical points of the action functional $\mathcal{E}_q \colon \H
\times \D \to \R$, defined as
\[
\mathcal{E}_q(u,\phi):=\frac 12 \irt |\n u|^2 -\frac 14 \irt |\n
\phi|^2 +\frac q2 \irt \phi u^2 -\irt G(u).
\]

The action functional $\E_q$ is strongly indefinite in the sense
that it is unbounded both from below and from above on infinite
dimensional subspaces. The indefiniteness can be removed using the
reduction method, by which we are led to study a one variable
functional that does not present such a strongly indefinite
nature. Indeed, it can be proved that $(u,\phi)\in H^1(\RT)\times
\D$ is a solution of \eqref{SP} (critical point of functional
$\mathcal{E}_q$) if and only if $u\in\H$ is a critical point of
the functional $J_q\colon \H\to \R$ defined as
\begin{equation*}%\label{eq:defI}
J_q(u)= \frac 12 \irt |\n u|^2 + \frac q4 \irt \phi_u u^2 -\irt
G(u),
\end{equation*}
and $\phi=\phi_u$.

Now, let ${\mathcal O}(2)$ denote the orthogonal group of the
rotation matrices in $\R^2$, that is
$$\OO(2)=\left\{\Big(
            \begin{array}{lr}
                \cos\alpha & -\sin\alpha\\
                \sin\alpha & \cos\alpha
            \end{array}
        \Big)\Big|\alpha\in[0,2\pi)\right\}.$$ For any $g\in\OO(2)$
define  the following action $\TT_g$ on $\H$:
$${\mathcal T}_g u(x)=
-u(\tilde g x)\in\H,\quad \tilde{g}=\Big(
            \begin{array}{lr}
                g & 0\\
                0 & -1
            \end{array}
        \Big).$$
        Now we set $$\Hco=\{u\in {\mathcal D}^1(\R^3,\R^3) \,|\, \TT_g u= u\;\;\forall g\in
{\mathcal
        O}(2)\}.$$
It is easy to see that $\Hco$ is the setting of the functions
cylindrically symmetric with respect to $(x_1,x_2)$ and odd with
respect to $x_3$.

Since $g$ is odd (and consequently $G$ is even) and since we have
that for any $u\in\H$ and $g\in\OO(2)$
    \begin{equation}\label{eq:sy}
        -{\mathcal T}_g \phi_u= \phi_{{\mathcal T}_g u}
    \end{equation}
by the Palais' symmetrical criticality principle we can prove that
$\Hco$ is a natural constraint for the action functional $J_q$
(see \cite{DA} for details).\\
We point out that, since $u\in\Hco$, we have that $\phi_u\in\Dce,$
the set of the functions in $\D$ that are cylindrically symmetric
with respect to the first two variables, and even with respect to
the third. To improve the notations, we will often use $r$ in the
place of $\sqrt {x_1^2+x_2^2}$.

We will proceed as follows: we consider the manifold
    \begin{equation}\label{eq:man}
        \M=\{u\in\Hco\mid \irt G(u)=1\}.
    \end{equation}
As proved in \cite{AP2} (see also \cite{BL1}), $\M$ is nonempty.
Consider indeed a family of functions
$\rho_R(r,x_3)=\xi\alpha_R(r)\beta_R(x_3),$ for $R>1,$ with

\begin{equation*}
\alpha_R(t):=\left\{
\begin{array}{ll}
1           &\hbox{\rm if}\quad |t|< R,
\\
R+1-|t|     &\hbox{\rm if}\quad R\le |t|< R+1,
\\
0           &\hbox{\rm otherwise},
\end{array}
\right.
\end{equation*}
and
\begin{equation*}
\beta_R(t):=\left\{
\begin{array}{ll}
0           &\hbox{\rm if}\quad 0< t \le 1\\
t-1         &\hbox{\rm if}\quad 1< t\le 2,
\\
1           &\hbox{\rm if}\quad 2< t\le R,
\\
R+1-t       &\hbox{\rm if}\quad R< t\le R+1,
\\
-\beta_R(-t)&\hbox{\rm if}\quad t\le 0.
\end{array}
\right.
\end{equation*}
We have $\rho_R\in\Hco,$ and for large $\bar R$
    $$\irt G(\rho_{\bar R})>0.$$
So, if $\sigma$ is a suitable rescaling parameter, the function
$$\rho_{\bar R,\sigma}:(r,x_3)\mapsto \rho_{\bar R}(\sigma r,\sigma x_3)$$
belongs to $\M.$

Then, we consider the functional
    \begin{equation}\label{eq:J}
        J_q(u)=\frac 1 2 \irt |\n u|^2 +\frac q 4 \irt \phi_uu^2
    \end{equation}
restricted on $\M,$ and we look for a minimizer $\bar u.$\\
Solving the minimizing problem, we find a Lagrange multiplier
$\l\in\R$ such that the tern $(\bar u, \phi_{\bar u}, \l)$ solves
the system
        \begin{equation*}
\left\{
\begin{array}{ll}
-\Delta u+q\phi u=\mu g(u)&\hbox{in }\RT,
\\
-\Delta \phi=q u^2&\hbox{in }\RT.
\end{array}
\right.
\end{equation*}
Then we apply the following
    \begin{theorem}\label{sol}
        Let $\bar u\in \M$ a minimizer for $J_q|_\M,$ and let $\l$ be
        the Lagrange multiplier. Then $\l$ is positive, and the couple
        $(\tilde u, \tilde \phi)\in \Hco\times\Dce$ defined
        rescaling as follows
            \begin{equation}\label{eq:realsol}
                \tilde u = \bar u(\cdot/\sqrt\l)\qquad\tilde\phi =
                \phi_{\bar u}(\cdot/\sqrt\l)
            \end{equation}
        solves the system
        \begin{equation}\label{eq:realsys}
\left\{
\begin{array}{ll}
-\Delta u+q'\phi u=g(u)&\hbox{in }\RT,
\\
-\Delta \phi=q' u^2&\hbox{in }\RT.
\end{array}
\right.
\end{equation}
        with $q'=q/\l.$
    \end{theorem}

\section{Compactness}\label{sec:comp}
In this section we present the main tool to get our result. We
first
need to introduce some notations and definitions.\\
Set $m_q=\inf_{u\in\M}J_q(u),$ and denote by $(u_n)_n:=(u^q_n)_n$
a sequence such that
    \begin{equation}\label{eq:min}
        u_n\in \M\qquad\hbox{and}\qquad J_q(u_n)\to m_q
    \end{equation}
and by $\phi_n=\phi_{u_n}$.

As in \cite{ADP,JL,K} we introduce the cut-off function $\chi\in
C^\infty(\R_+,\R)$ satisfying
\begin{equation}\label{eq:defchi}
    \left\{
\begin{array}{ll}
    \chi(s)=1,&\hbox{for }s\in[0,1],\\
    0\le \chi(s)\le 1,&\hbox{for }s\in]1,2[,\\
    \chi(s)=0,&\hbox{for }s\in[2,+\infty[,\\
    \|\chi '\|_\infty \le 2,&
\end{array}
    \right.
\end{equation}
and, for every $T>0$, we denote
\[
k_T(u)=\chi\left(\frac{\|u\|^2}{T^2} \right).
\]

Moreover, assume the following definitions
    \begin{align*}
        \Jtq(u)      &=\frac 1 2 \irt |\n u|^2 +\frac q 4 k_T(u) \irt
        \phi_uu^2\\
        \mu_n^{T,q}(\O)&=\frac 12 \int_\O|\n u_n|^2+\int_\O
        G_2(u_n)+\frac q 4 k_T(u_n)\int_\O \phi_n
        u_n^2,
    \end{align*}
where $\Omega\subset\R^3.$ Set also $\mtq=\inf_{u\in\M}\Jtq(u),$
and denote by $(u^{T,q}_n)_n$ a minimizing sequence of $\Jtq|_\M.$
It is trivial to see that $\mtq\le m_q\le m_{\bar q}$ for any
$T>0$ and any $q\le \bar q.$
    \begin{lemma}
        For any $T,q>0$ the measures $\mu_n^{T,q}$ are positive and bounded, i.e.
        $(\mu_n^{T,q}(\RT))_n$ is bounded. Moreover $\mu_n^{T,q}$
        is bounded $T-$uniformly.
    \end{lemma}
    \begin{proof}
        The positiveness is a trivial consequence of the
        definition of the measures.\\
        As to boundedness, by the very definition of $u_n$
        we have only to check if $(\irt G_2(u_n))_n$ is
        bounded. But by \eqref{eq:G1G2} we have
            \begin{equation}\label{eq:com}
                1+\irt G_2(u_n)=\irt G_1(u_n)\leq \irt
                \eps G_2(u_n)+C\irt|u_n|^6
            \end{equation}
        and then
            \begin{equation}\label{eq:com2}
                1+(1-\eps)\irt G_2(u_n)\leq C'\left(\int_\O|\n
                u_n|^2\right)^3
            \end{equation}
        for $0<\eps <1$ and $C,$ $C'$ suitable positive
        constants.\\
        The $T-$uniform boundedness is a consequence of the fact
        that for any $n\ge 1$ and for any $T>0$ $k_T(u_n)\leq 1.$
    \end{proof}
Let $c=c^T_q$ be the limit (up to a subsequence) of
$\mutq_n(\RT)$. Of course $c>0$ because, otherwise, we would
contradict \eqref{eq:com2}.

    \begin{lemma}\label{le:cTq}
        For any $\bar q$ there exists $\bar T$ such that
            \begin{equation}\label{eq:limsup}
                \limsup_n\|u^q_n\|\leq  T,\quad\limsup_n\|u_n^{T,q}\|\le T
            \end{equation}
        for all $q\le\bar q$ and $T\ge\bar T.$\\
        As a consequence, every a minimizing sequence for
        $J_q|_\M,$ is a minimizing sequence also for $\Jtq|_\M.$
    \end{lemma}
    \begin{proof}
        Fix $\bar q>0$ and $q\leq \bar q$ and consider a minimizing sequence $u_n=u^q_n$
        as in \eqref{eq:min}. Consider also $\bar T>0$ whose precise
        estimate will be given later, $T\ge\bar T$ and $(u_n^{T,q})_n$ a minimizing sequence of
        $\Jtq|_\M.$
        Certainly we have that
            \begin{equation}\label{eq:gr}
                \irt|\n u_n|^2\le 2 m_q + o_n(1)\le 2 m_{\bar q} + o_n(1).
            \end{equation}
        By \eqref{eq:G2} and \eqref{eq:com2} we have also
            \begin{align}\label{eq:l2}
                \irt |u_n|^2&\le \frac\o 2\irt G_2(u_n)\le C\big(\irt|\n
                u_n|^2\big)^3\nonumber\\
                &\le C' (2m_q + o_n(1))^3= 8
                C'm_q^3+o_n(1)\le 8
                C'm_{\bar q}^3+o_n(1).
            \end{align}
        Since $m^T_q\le m_q,$ the same estimates can be proved also for
        $(u_n^{T,q})_n.$
        By \eqref{eq:gr} and \eqref{eq:l2} we conclude the first part of the
        proof taking $\bar T>\max(2m_{\bar q},8C'm_{\bar q}^3).$

        To prove the final part of the theorem, it is sufficient
        to show that $m^T_q=m_q.$ But for a sufficiently large $\nu\ge 1$ and any $n\ge\nu,$
        by \eqref{eq:limsup} we have that $k_T(u_n^{T,q})=1$ and
        $\Jtq(u_n^{T,q})=J_q(u_n^{T,q})\ge m_q.$ We deduce that $m^T_q\ge
        m_q$ and then $m^T_q=m_q.$

    \end{proof}

By the concentration and compactness principle (see \cite{L1}),
one of the following holds:
\begin{description}
\item[\;\;{\it vanishing}\,{\rm :}] for all $R>0$
\[
\lim_n \sup_{\xi \in \RT}\int_{B_R(\xi)} d \mutq_n =0;
\]
\item[\;\;{\it dichotomy}\,{\rm :}] for a subsequence of
$(\mutq_n)_n,$ there exist a constant $\tilde c\in (0, c)$, $R>0$,
two sequences $(\xi_n)_n$ and $(R_n)_n$, with $R\le R_n$ for any
$n$ and $R_n\to +\infty,$ such that
\begin{equation}\label{eq:dic}
\int_{B_R(\xi_n)} d \mutq_n \to \tilde c,\;\; \int_{\RT \setminus
B_{R_n}(\xi_n)} d \mutq_n \to c -\tilde c,
\end{equation}
\item[\;\;{\it compactness}\,{\rm :}] there exists a sequence
$(\xi_n)_n$ in $\RT$ with the following property: for any $\d>0$,
there exists $r=r(\d)>0$ such that
\[
\int_{B_R(\xi_n)} d \mu_n \ge c -\d.
\]
\end{description}

    \begin{theorem}\label{th:nonv}
        Vanishing does not occur
    \end{theorem}
    \begin{proof}

Suppose by contradiction, that for all $R>0$
\[
\lim_n \sup_{\xi \in \RT}\int_{B_R(\xi)} d \mu^{T,q}_n =0.
\]
In particular, we deduce that there exists $\bar R>0$ such that
\begin{equation*}
\lim_n \sup_{\xi \in \RT}\int_{B_{\bar R}(\xi)} u_n^2=0.
\end{equation*}
By this and Lemma \ref{le:cTq}, we have that $u_n\to 0$ in
$L^s(\RT),$ for $2< s <6$ (see \cite[Lemma I.1]{L2}). As a
consequence, since $(u_n)_n \subset \M$ and by \eqref{eq:G1G2b},
we get for $0<\eps<1$ and $C'_\eps>0$
            \begin{equation*}
                1+\irt G_2(u_n)=\irt G_1(u_n)\leq \irt
                \eps G_2(u_n)+C'_\eps\irt|u_n|^{p+1}
            \end{equation*}
        and then
            \begin{equation*}
                1+(1-\eps)\irt G_2(u_n)\leq
                C'_\eps\irt|u_n|^{p+1}\to 0.
            \end{equation*}
    \end{proof}
        From now on, if the
        notation of a ball does not present explicitly expressed the center, than we assume
        it is the origin.
    \begin{theorem}\label{th:conc}
        For any $\bar q>0,$ there exist $\bar T>0$ such that for
        any $T\ge \bar T$ and a suitable $0<q(T)\le\bar q,$ either $\mu^{T,q(T)}_n$ concentrates in a ball
        $B_R$ (namely compactness holds for $\xi_n=(0,0,0),$ $n\ge 1$)
        or it exhibits the following dichotomic behaviour: there exist $R>0$ and
        a divergent sequence $\xi_n=(0,0,x^n_3)_n$ in $\R^3$ such that
            \begin{align*}
                \int_{B_R(\xi_n)} d \mutq_n &\to \frac c 2\\
                \int_{B_R(-\xi_n)} d \mutq_n &\to \frac c 2.
            \end{align*}
    \end{theorem}
    \begin{proof}
        Take $\bar q>0,$ and let $\bar T>0$ be as in Lemma \ref{le:cTq}.\\
        Set $T\ge \bar T.$
        Suppose that dichotomy holds and let $\tilde c\in (0, c)$, $R>0$, $(\xi_n)_n,$ $(R_n)_n$ be as in the
        dichotomy hypothesis.
        We prove that $(\xi_n)_n$ is bounded with respect to the
        first two variables. Otherwise, we should have $\xi_n\simeq(r_n,x^n_3)$
        with $r_n\to+\infty$ and
            \begin{equation}\label{eq:toc}
                \int_{B_R(\xi_n)} d \mutq_n = \tilde c + o_n(1).
            \end{equation}
        We deduce that there exists a positive constant $C>$ such that
            \begin{equation*}
                \int_{B_R(\xi_n)} |\n u_n|^2 + \int_{B_R(\xi_n)}G_2(u_n) \ge C
            \end{equation*}
        (otherwise, by \eqref{eq:phiq} and \eqref{eq:G2}, we would get a contradiction with \eqref{eq:toc}).
        But, for $r_n$ that goes to infinity, the set $B_{\xi_n+ R}(0)\setminus B_{\xi_n - R}(0)$ contains
        an increasing number of disjoint balls of the type $B_R(r',x_3^n)$, with $r_n=r':=\sqrt{(x_1')^2+(x_2')^2}$
        and, by the symmetry properties on $u_n$, for any $n\ge
        1,$
            \begin{equation*}
                \int_{B_R((r',x^3_n))} |\n u_n|^2 +
                \int_{B_R((r',x^3_n))}G_2(u_n)=
                \int_{B_R(\xi_n)} |\n u_n|^2 +
                \int_{B_R(\xi_n)}G_2(u_n).
            \end{equation*}
        As a consequence, we would have that
            \begin{equation*}
                \irt |\n u_n|^2 + \irt G_2(u_n) \to +\infty
            \end{equation*}
        that, taking \eqref{eq:com2} into account, brings a contradiction to Lemma
        \ref{le:cTq}.

        By the boundedness of $(\xi_n)_n$ with respect to $r_n$, it is not restrictive
        to suppose that such a sequence belongs to the
        $x_3-$axis. Indeed, for any $n\ge 1,$ the ball $B_R(\xi_n)$ is
        contained in $B_{R'}((0,0,x_3^n))$, where $R'=R+\sup_{n}|r_n|$.

        Now we consider the following possibilities:
            \begin{itemize}
                \item $(x_n^3)_n$ is bounded
                \item $(x_n^3)_n$ is unbounded.
            \end{itemize}
        If $(x_n^3)_n$ is bounded, all the balls of the type
        $B_R(\xi_n)$ are contained in $B_{R''}$, where $R''=R'+\sup_{n}|x_n^3|$.
        Replacing $R'$ by $R'',$ we have that
            \begin{equation}
                \int_{B_{R''}(0)} d \mutq_n = \tilde c + o_n(1).
            \end{equation}
        Consider a sequence of radially symmetric
        cut-off functions $\rho_n \in C^1(\RT)$ such that $\rho_n\equiv 1$ in
        $B_{R}(0)$, $\rho_n\equiv 0$ in $\RT \setminus
        B_{R_n}(0)$, $0\le \rho_n\le 1$ and $|\n \rho_n|\le 2/(R_n-R)$.
        \\
        We set
        \[
        v_n:=\rho_n u_n,\qquad w_n:=(1-\rho_n)u_n.
        \]
        Certainly $v_n$ and $w_n$ are in $\Hco$ and
            \begin{align}
                \|v_n\|&\le\|u_n\|+o_n(1)\label{eq:nor1}\\
                \|w_n\|&\le\|u_n\|+o_n(1).\label{eq:nor2}
            \end{align}
        If we denote $\O_n:=B_{R_n}\setminus B_R,$
        by dichotomy hypothesis we deduce that
            \begin{equation}\label{eq:to0}
                \int_{\O_n}|\n u_n|^2\to 0,\quad \int_{\O_n} G_2(u_n)\to 0,\quad
                \int_{\O_n} \phi_n u_n^2\to 0,
            \end{equation}
        and, in particular,
            \begin{equation}\label{eq:un0}
                \|u_n\|_{H^1(\O_n)}\to 0.
            \end{equation}
        Since for suitable $\eps,$ $C_\eps,$ and $C'>0$
            \begin{align}\label{eq:less}
                \int_{\O_n}G_1(u_n)&\le \eps \int_{\O_n}G_2(u_n)+ C_\eps \int_{\O_n}|u_n|^{p+1}\nonumber\\
                                   &\le \eps \int_{\O_n}G_2(u_n)+ C'\|u_n\|^{p+1}_{H^1(\O_n)},
            \end{align}
        we have also that
            \begin{equation}\label{eq:to00}
                \int_{\O_n}G_1(u_n)\to 0.
            \end{equation}
        Since by simple computations, using \eqref{eq:to0}, we have
            \begin{equation*}
                \int_{\O_n}|\n v_n|^2\to 0\hbox{ and } \int_{\O_n}|\n w_n|^2\to 0,
            \end{equation*}
        we easily infer that
            \begin{equation}\label{eq:gr0}
                \irt |\n u_n|^2 = \irt |\n v_n|^2 + \irt |\n w_n |^2 + o_n(1).
            \end{equation}
        Moreover we can prove also that
            \begin{align}
                \int_{\O_n} G_1(v_n)\to 0\quad &\int_{\O_n} G_2(v_n)\to 0\label{eq:to000}\\
                \int_{\O_n} G_1(w_n)\to 0\quad &\int_{\O_n} G_2(w_n)\to 0\label{eq:to0000}
            \end{align}
        Indeed, by \eqref{eq:G1G2a}, the growth conditions on $g$ and \eqref{eq:un0},
            \begin{align*}
                \int_{\O_n} G_1(v_n)&\le C(\int_{\O_n} |v_n|^2 + \int_{\O_n}|v_n|^{p+1})\\
                &\le C'(\|v_n\|_{H^1(\O_n)}^2+\|v_n\|_{H^1(\O_n)}^{p+1})\\
                &\le C'(\|u_n\|^2_{H^1(\O_n)}+\|u_n\|^{p+1}_{H^1(\O_n)}+o_n(1))=o_n(1)\\
                \int_{\O_n} G_2(v_n)&=-\int_{\O_n} G(v_n) +
                \int_{\O_n} G_1(v_n)\\
                &\le
                C(\int_{\O_n} |v_n|^2 + \int_{\O_n}|v_n|^{p+1})\le C'(\|v_n\|_{H^1(\O_n)}^2+\|v_n\|_{H^1(\O_n)}^{p+1})\\
                &\le C'(\|u_n\|_{H^1(\O_n)}^2+\|u_n\|_{H^1(\O_n)}^{p+1}+o_n(1))\le o_n(1),
            \end{align*}
        and we proceed analogously for $w_n.$
        By \eqref{eq:to0}, \eqref{eq:to00}, \eqref{eq:to000} and \eqref{eq:to0000}, we deduce that
            \begin{equation}
                \irt G_i(u_n)=\irt G_i(v_n) + \irt G_i(w_n) + o_n(1),\quad i=1,2.\label{eq:Gi0}
            \end{equation}
        Finally, as in \cite{AP08}, we have
            \begin{equation}\label{eq:phi0}
                \irt\phi_n u_n^2\ge\irt
                \phi_{v_n}v_n^2+\irt\phi_{w_n}w_n^2+o_n(1).
            \end{equation}
        By \eqref{eq:gr0}, \eqref{eq:Gi0} and \eqref{eq:phi0},
        taking into account that by
        \eqref{eq:nor1}, \eqref{eq:nor2} and Lemma \ref{le:cTq} we have
        $1=k_T(u_n)=k_T(v_n)=k_T(w_n),$ we deduce that
            \begin{equation*}
                \mtq=\Jtq(u_n)+o_n(1)\ge \Jtq (v_n) + \Jtq (w_n)+o_n(1)
            \end{equation*}
        and, as a consequence,
            \begin{align}
                \Jtq(v_n)&\to\widetilde \mtq\nonumber\\
                \Jtq(w_n)&\to\overline \mtq\label{eq:bah}
            \end{align}
        with $\widetilde \mtq+\overline \mtq\le\mtq.$

        For the moment, we assume that $\widetilde \mtq\neq 0$ and $\overline \mtq\neq 0.$

        We have to consider the following possibilities
            \begin{enumerate}
                \item[\it i)] {\it there exists $0<\l<1$ such that, up to subsequences,}
                    \begin{align*}
                        \irt G(v_n)&\to\l\\
                        \irt G(w_n)&\to1-\l.
                    \end{align*}

                Consider the rescaled functions so defined:
                $\tilde v_n(\cdot)= v_n(\sqrt[3]{\l}\,\cdot)$ and $\tilde w_n(\cdot)=
                w_n(\sqrt[3]{1-\l}\,\cdot)$ so that we respectively have
                    \begin{align*}
                        \Jtq(\tilde v_n)&\ge \mtq +o_n(1)\\
                        \Jtq(\tilde w_n)&\ge \mtq +o_n(1).
                    \end{align*}
                The following chain of inequalities holds
                    \begin{align}\label{eq:ch}
                        o_n(1) + \widetilde\mtq &= \Jtq(v_n)\nonumber\\
                        &= \frac {\sqrt[3]\l} 2\irt |\n\tilde v_n|^2\nonumber\\
                        &\qquad+\frac{q(\sqrt[3]\l)^5}{4}\chi\left(\frac {(\sqrt[3]\l)^2\|\n\tilde v_n\|_2^2+
                        \l^2\|\tilde v_n\|_2^2}{T^2}\right)\irt\phi_{\tilde v_n}\tilde v_n^2\nonumber\\
                        &\ge \frac {\sqrt[3]\l} 2 \irt |\n\tilde v_n|^2+ \frac{q(\sqrt[3]\l)^5}{4}\chi\left(
                        \frac{\|\tilde v_n\|^2}{T^2}\right)\irt\phi_{\tilde v_n}\tilde v_n^2\nonumber\\
                        &\ge \l \mtq + \left(\frac {\sqrt[3]\l-\l} 2\right)
                        \irt |\n\tilde v_n|^2\nonumber\\
                        &\qquad+ \frac{q(\sqrt[3]{\l^5}-\l)}{4}\chi\left(
                        \frac{\|\tilde v_n\|^2}{T^2}\right)\irt\phi_{\tilde v_n}\tilde v_n^2+o_n(1).
                    \end{align}
                Now observe that, since $\irt G(\tilde v_n)\to 1,$ computing as in \eqref{eq:com}
                and \eqref{eq:com2},
                    \begin{equation}\label{eq:boubel}
                        o_n(1)+ 1+(1-\eps)\irt G_2(\tilde v_n)\le C\left(\irt|\n\tilde v_n|^2\right)^3,
                    \end{equation}
                for $0< \eps <1,$ we deduce that $\|\n\tilde v_n\|_2$ is bounded below by a positive constant.
                Moreover, by \eqref{eq:phiq},
                    \begin{equation}\label{eq:bouup}
                        \chi\left(
                        \frac{\|\tilde v_n\|^2}{T^2}\right)\irt\phi_{\tilde v_n}\tilde v_n^2\le q C T^4,
                    \end{equation}
                so by \eqref{eq:ch}, \eqref{eq:boubel} and \eqref{eq:bouup}, for suitable
                $a,b>0$, we have
                    \begin{equation}\label{eq:comw}
                        \widetilde\mtq\ge \l \mtq + a (\sqrt[3]\l-\l)
                        +bq^2(\sqrt[3]{\l^5}-\l)T^4.
                    \end{equation}
                But
                    \begin{align*}
                        a (\sqrt[3]\l-\l)
                        +bq^2(\sqrt[3]{\l^5}-\l)T^4 &=
                        \sqrt[3]\l(1-\sqrt[3]{\l^2})(a-bq^2\l T^4)\\
                                                    &\ge \sqrt[3]\l(1-\sqrt[3]{\l^2})(a-bq^2T^4),
                    \end{align*}
                so, if we take $q<\sqrt {\frac a {bT^4}}$, from \eqref{eq:comw} we obtain
                $\widetilde\mtq>\l\mtq.$\\
                Repeating the same computations with $\tilde w_n$ in the place
                of $\tilde v_n,$ we can prove that $\overline\mtq>(1-\l)\mtq$. Summing up, we get
                    $$\mtq\ge\widetilde \mtq +\overline\mtq>\l\mtq+(1-\l)\mtq=\mtq$$
                and then a contradiction.
                \item[\it ii)] {\it there exists $\l\ge 1$ such that, up to subsequences,}
                    \begin{align*}
                        \irt G(v_n)&\to\l\\
                        \hbox{or}\\
                        \irt G(w_n)&\to\l.
                    \end{align*}
                Suppose that the first holds, and set $\l_n=\irt G(v_n)$ and $\tilde v_n=v_n(\sqrt[3]{\l_n}\,\cdot)\in\M$.
                We would have the following chain of inequalities
                    \begin{align}\label{eq:twice}
                        \mtq &\le \Jtq (\tilde v_n)=\frac1{2\sqrt[3]{\l_n}}\irt |\n v_n|^2+\frac q {4\sqrt[3]
                        {\l_n^5}}k_T(\tilde v_n)\irt\phi_{v_n}
                        v_n^2\nonumber\\
                        &\le\frac 1 2 \irt |\n v_n|^2 +  \frac q 4 k_T(v_n)\irt
                        \phi_{v_n}v_n^2 \to \widetilde\mtq < \mtq,
                    \end{align}
                where we have used the fact that $\|\tilde
                v_n\|^2\le\|v_n\|^2\le \|u_n\|^2 + o_n(1) < T^2$ to deduce that
                $k_T(\tilde v_n)= k_T(v_n)=1.$

                Now we remove the assumption that $\widetilde \mtq\neq
                0$ and $\overline \mtq\neq 0.$
                If, for instance, $\overline\mtq=0$, from \eqref{eq:dic} and \eqref{eq:bah} we would deduce that
                    \begin{align*}
                        \irt |\n w_n|^2 &\to 0\\
                        \irt G_2(w_n)&\ge \a + o_n(1)
                    \end{align*}
                with $\a >0.$\\
                Hence, by \eqref{eq:G1G2b}, for any $\eps>0$ we have
                    \begin{align*}
                        \irt G_1(w_n) &< \eps \irt G_2(w_n) +
                        C_\eps\irt |\n w_n|^2\\
                        & = \eps \a + C_\eps o_n(1) +
                        o_n(1),
                    \end{align*}
                and then $\irt G_1(w_n)\to0.$ So
                    \begin{align*}
                        1&=\irt G(u_n)=\irt G(v_n) + \irt G(w_n) +
                        o_n(1)\\
                         &=\irt G(v_n) - \irt G_2(w_n) + o_n(1)\\
                         &\le\irt G(v_n) - \a + o_n(1)
                    \end{align*}
                which implies that, up to subsequences, $\irt
                G(v_n)\to\l>1.$\\
                As in \eqref{eq:twice},
                    \begin{align*}
                        \mtq &\le \liminf_n\left(\frac1{2\sqrt[3]{\l_n}}\irt |\n v_n|^2+\frac q {4\sqrt[3]
                        {\l_n^5}}k_T(\tilde v_n)\irt\phi_{v_n}
                        v_n^2\right)\\
                             &\le \liminf_n\left(\frac1{2\sqrt[3]{\l}}\irt |\n v_n|^2+\frac q {4\sqrt[3]
                        {\l^5}}k_T(v_n)\irt\phi_{v_n}
                        v_n^2\right) \\
                             &< \liminf_n\left(\frac1{2}\irt |\n v_n|^2+\frac q {4}k_T(v_n)\irt\phi_{v_n}
                        v_n^2\right)=\widetilde\mtq=\mtq
                    \end{align*}
                and then a contradiction. The case $\widetilde
                \mtq=0$ is analogous.
            \end{enumerate}

                We have showed that, in any case, if $(x_3^n)_n$
                is bounded, dichotomy leads to a contradiction.\\
                It remain to study what would happen if
                $(x_3^n)_n$ was unbounded. Suppose that the
                dichotomic behaviour of the statement does not
                hold. Then, by the evenness of the functional and the oddness with
                respect to the third variable of the functions in $\Hco$,
                    \begin{align*}
                        \int_{B_R(\xi_n)} d \mutq_n &\to \tilde c <\frac c 2\\
                        \int_{B_R(-\xi_n)} d \mutq_n &\to \tilde c <\frac c 2\\
                        \int_{\RT\setminus\S_n} d \mutq_n &\to
                        c-2\tilde c
                    \end{align*}
                where we have assumed the following notation: $\S_n=B_{R_n}(\xi_n)\cup B_{R_n}(-\xi_n).$
                Observe that we can redefine the sequence $R_n$ in such a way
                we have $B_{R_n}(\xi_n)\cap
                B_{R_n}(-\xi_n)=\emptyset.$\\
                Now, consider a sequence of $\xi_n-$radially symmetric
                cut-off functions $\rho_n \in C^1(\{x\in\RT\mid x_3>0\})$ such that $\rho_n\equiv 1$ in
                $B_{R}(\xi_n)$, $\rho_n\equiv 0$ in $\{x\in\RT\mid x_3>0\} \setminus
                B_{R_n}(0)$, $0\le \rho_n\le 1$ and $|\n \rho_n|\le
                2/(R_n-R)$, and define $\sigma_n\in C^1(\RT)$ by
                evenness with respect to the third variable.\\
                Set $v_n=\sigma_n u_n$ and $w_n=(1-\sigma_n) u_n.$
                Of course $v_n$ and $w_n$ are in $\Hco$ and we can repeat
                exactly the same arguments as in the {\it $x_3^n$
                bounded case} to get a contradiction.

                The proposition is so completely proved.
            \end{proof}

    \section{Proof of the main Theorem}\label{sec:proof}

        From now on, all the sequences considered have their $\limsup$ in the
        norm of $\H$ less than $\bar T,$ being $\bar T$ the same as in Lemma
        \ref{le:cTq}. Therefore there is no difference between
        $J_q$ and $J^T_q$ evaluated on them.
            \begin{theorem}\label{th:inf}
                Let $q$ be as in Theorem \ref{th:conc}, then the infimum $m_q$ is achieved.
            \end{theorem}
            \begin{proof}
                Suppose that the dichotomy situation described in
                Theorem \ref{th:conc} holds. Since $x_3^n\to+\infty,$
                we can suppose that for any $n\ge 1$ we have
                $x_3^n>3R.$ Then, consider a sequence of $\xi_n-$radially symmetric
                cut-off functions $\rho_n \in C^1(\{x\in\RT\mid x_3>0\})$ such that $\rho_n\equiv 1$ in
                $B_{R}(\xi_n)$, $\rho_n\equiv 0$ in $\{x\in\RT\mid x_3>0\} \setminus
                B_{2R}(\xi_n)$, $0\le \rho_n\le 1$ and $|\n \rho_n|\le
                2/R$, and define $\sigma_n\in C^1(\RT)$ by
                evenness with respect to the third variable.\\
                Set $v_n=\s_n u_n\in\Hco$ and for any $x=(x_1,x_2,x_3)\in\R^3$ define
                    \begin{equation}
                        \tilde v_n(x)=\left\{
                            \begin{array}{ll}
                                v_n(x_1,x_2,x_3+\xi_n-3R)&\hbox{if
                                } x_3>0
                                \\
                                v_n(x_1,x_2,x_3-\xi_n+3R)&\hbox{if
                                } x_3<0.
                            \end{array}
                        \right.
                    \end{equation}
                We would
                have that, for $R'=4 R,$
                    \begin{equation}
                        \frac 12 \int_{B_{R'}}|\n \tilde v_n|^2+\int_{B_{R'}}
                        G_2(\tilde v_n)+\frac q 4 \int_{B_{R'}} \phi_{\tilde v_n} \tilde
                        v_n^2\to c
                    \end{equation}
                and it is easy to verify also that a sequence so
                defined is such that
                $$\irt G(\tilde v_n)\to 1\quad\hbox{and}\quad J_q(\tilde v_n)\to m_q.$$
                So, in any case, by Theorem \ref{th:conc} we are
                able to obtain a minimizing sequence that we label $(u_n)_n$ for the
                functional restricted to $\M,$ which concentrates
                on a ball centered at the origin and with a sufficiently
                large radius.\\
                By boundedness of the sequence, we can extract a
                subsequence weakly convergent in $H^1-$norm to a
                function $u.$\\
                As a consequence of the weak convergence, the
                Fatou lemma and the weak lower semicontinuity
                of $\|\n\cdot\|_2,$, we have
                    \begin{equation}\label{eq:liminf}
                        J_q(u)\le\liminf_n J_q(u_n)=m_q.
                    \end{equation}
                Since we also have
                    \begin{align}
                        u_n & \to u \hbox { pointwise}\label{eq:point}\\
                        u_n & \to u \hbox { in } L^q(B), \hbox{ for
                        any bounded set } B \hbox { and any } q\in
                        [1,6[,
                    \end{align}
                we deduce that $u\in\Hco\setminus\{0\}$ and $G_1(u_n(x))\to
                G_1(u(x))$ for any $x\in\RT.$\\
                Since
                    $$G_1(s)= o_n(s^2 + |s|^{p+1})\hbox{ for } s\to 0
                    \hbox{ and } s\to\infty,$$
                and by concentration we have
                    $$\int_{\RT\setminus B_R} u_n^2 + |u_n|^{p+1}\to 0,$$
                by standard compactness argument (see for
                instance the proof of Theorem A.I. in
                the Appendix in \cite{BL1}) we deduce that
                    \begin{equation*}
                        \irt G_1(u_n)\to\irt G_1(u).
                    \end{equation*}
                On the other hand, we also have that
                    \begin{equation*}
                        1+\irt G_2(u_n) = \irt G_1(u_n)\to\irt
                        G_1(u)
                    \end{equation*}
                and then, by \eqref{eq:point}
                    \begin{equation*}
                        \irt G_2(u)\le\liminf_n\irt G_2(u_n)= \irt
                        G_1(u) -1.
                    \end{equation*}
                that is $\irt G(u)\ge 1.$ We deduce that $\irt G(u)=1,$ otherwise we set
                $\bar u=u(K\,\cdot)\in\M$ with $K=\sqrt[3]{\irt G(u)}>1$
                and by \eqref{eq:liminf} we have,
                    \begin{align*}
                        m_q & \le J_q(\bar u)=\frac 1 {2\sqrt[3]
                        K}\irt |\n u|^2+\frac q
                        {4\sqrt[3]{K^5}}\irt \phi_u u^2\\
                            & < J_q(u)\le m_q
                    \end{align*}
                which is a contradiction.\\
                So $\irt G(u)=1$, and by \eqref{eq:liminf}
                $J_q(u)=m_q.$
            \end{proof}
        \begin{proofsol}
            Let $\bar u\in\M$ be such that $J_q(\bar u)=m_q$ and
            let $\l\in\R$ be the Lagrange multiplier.
            To show that $\l >0,$ we can proceed as in \cite[pg 327]{BL1}. Now
            define $\tilde u$ and $\tilde \phi$ as in
            \eqref{eq:realsol}. We prove that $(\tilde u,\tilde\phi)$
            satisfies the second equation of the system
            \eqref{eq:realsys}
                \begin{align*}
                    -\Delta \tilde \phi&=-\frac 1
                    \l\Delta{\phi_{\bar u}}(\cdot/\sqrt\l)\\
                    &=\frac 1
                    \l q {\bar u}^2(\cdot/\sqrt\l)\\
                    &=q'{\bar u}^2(\cdot/{\sqrt \l})=q'\tilde u^2.
                \end{align*}
            We prove that $(\tilde u,\tilde\phi)$
            satisfies the first equation of the system
            \eqref{eq:realsys}
                \begin{align*}
                    -\Delta \tilde u & = -\frac 1 \l \Delta \bar u
                    (\cdot/\sqrt\l)\\
                    &= -\frac 1 \l q \phi_{\bar
                    u}(\cdot/\sqrt\l)\bar u (\cdot/\sqrt\l) +
                    g(\bar u(\cdot/\sqrt\l))\\
                    &= - q'\tilde \phi\tilde u + g(\tilde u)
                \end{align*}
        \end{proofsol}
        \begin{proofmain}
            Let $(u,\phi)$ be a solution found by Theorem \ref{sol}.
            The symmetry properties derive from the natural constraint where
            we have studied the functional of the action and
            \eqref{eq:sy}.\\
            Now, observe that $u$ can be assumed
            nonnegative in the semispace $x_3>0$ and nonpositive
            in the semispace $x_3<0$.\\
            In fact, if $\bar u$ is a minimizer obtained as in
            Theorem \ref{th:inf}, we can replace it with the
            function
                \begin{equation*}
                    v=\left\{
        \begin{array}{ll}
                |\bar u| &\hbox{ on } \R^2\times ]0,+\infty[;
                \\
                -|\bar u| &\hbox{ on } \R^2\times ]-\infty,0[.
      \end{array}
      \right.
                \end{equation*}
      Obviously $v\in\Hco$ and since $J_q$ and $G$ are even, $v$
      is also a minimizer of $J_q|_\M.$\\
      Now we can apply the strong
      maximum principle in the second equation, and obtain that $\phi>0$, and in the first equation,
      obtaining that $u$ can
      vanish only on the plane $x_3=0.$ The same considerations on the sign hold for $(\tilde u,\tilde\phi),$
      and are true everywhere, since by a standard regularity argument, we
      can prove that $\tilde u$ and $\tilde\phi$ are in
      $C^{2,\a}_{loc}(\RT),$ with $\a\in (0,1).$

        \end{proofmain}

\end{document}